\documentclass[12pt]{amsart}
\usepackage{amsmath, amsfonts, amsbsy, amsthm, amscd, graphicx}
\usepackage{amssymb, latexsym, color}

\numberwithin{equation}{section}

\theoremstyle{plain}
\newtheorem{Theorem}{Theorem}[section]
\newtheorem{Proposition}[Theorem]{Proposition}

\newtheorem{Corollary}[Theorem]{Corollary}

\newtheorem{Problem}[Theorem]{Problem}

\theoremstyle{definition}

\theoremstyle{remark}
\newtheorem{Remark}{{\bf Remark}}
\newtheorem{Example}{Example}

\newcommand{\R}{{\mathbb R}}

\newcommand{\F}{{\mathbb F}}
\newcommand{\Z}{{\mathbb Z}}

\newcommand{\bary}{\mathop{\rm bar}\nolimits}

\begin{document}

\title{Optimal embedding and spectral gap of a finite graph}

\author{Takumi Gomyou$\mbox{}^1$} 
\thanks{$\mbox{}^{1,4}$Graduate School of Mathematics, Nagoya University,
Chikusa-ku, Nagoya 464-8602, Japan, d16001m@math.nagoya-u.ac.jp/ 
nayatani@math.nagoya-u.ac.jp}

\author{Toshimasa Kobayashi$\mbox{}^2$} 
\thanks{$\mbox{}^2$Faculty of Science and Engineering, Setsunan University, 
17-8 Ikedanaka-machi, Neyagawa-city, Osaka, 572-8508, Japan, kobayashi@mpg.setsunan.ac.jp}

\author{Takefumi Kondo$\mbox{}^3$}
\thanks{$\mbox{}^3$Mathematics and Computer Science, Graduate School of Scinece and Engineering, Kagoshima University, 
1-21-35 Korimoto, Kagoshima-city, Kagoshima, 890-0065, Japan, takefumi@sci.kagoshima-u.ac.jp}

\author{Shin Nayatani$\mbox{}^4$}


\maketitle

\begin{abstract}
We introduce a new optimization problem regarding embeddings of a graph into a Euclidean space 
and discuss its relation to the two, mutually dual, optimizations problems introduced 
by G\"oring-Helmberg-Wappler. 
We prove that the Laplace eigenvalue maximization problem of G\"oring et al is also  dual to our embedding 
optimization problem. 
We solve the optimization problems for generalized polygons and graphs isomorphic 
to the one-skeltons of regular and semi-regular polyhedra. 
\end{abstract} 

\section*{Introduction} 
In this paper, we introduce a new optimization problem regarding embeddings 
of a finite graph into a Euclidean space, motivated by the study of a certain invariant 
of the metric cone over a CAT($1$) metric graph. 
The problem is related to the maximization problem regarding the first nonzero eigenvalue 
of the Laplacian, introduced by G\"oring-Helmberg-Wappler \cite{GoringHelmbergWappler2}. 
We discuss a relation between these two problems. 
In particular, we establish an inequality relating the optimal values of the problems 
and also give an example for which the equality sign is attained. 

A similar optimization problem regarding graph-embeddings was also considered 
in \cite{GoringHelmbergWappler2}. 
The problem is dual to their eigenvalue maximization problem mentioned above, 
and more remarkably there is no duality gap, meaning that the optimal values of the 
two problems necessarily coincide. 
We discuss relation between two optimization problems regarding graph-embeddings,  
and find a precise relation between the optimal values of these problems. 
This relation, combined with the no-duality-gap result mentioned above, 
makes it possible to establish a formula computing an optimal value of our 
embedding optimization problem in terms of that of the eigenvalue maximization 
problem. 

We give examples of graphs for which the optimization problems due to 
G\"oring-Helmberg-Wappler can be explicitly solved. 
They are isomorphic to the one-skeltons of regular and semi-regular polyhedra, 
and the optimal solutions for the embedding optimization problem realize 
the graphs as the one-skeltons of the given polyhedra. 

A dual problem in the framework of semidefinite programing can be formulated 
if a primal problem and an appropriate Lagrange function are given, and 
different choices of Lagrange function may produce different dual problems. 
In fact, we prove that the eigenvalue maximization problem is also dual to 
our embedding optimization problem. 

\section{Embedding and spectrum gap of a finite graph} 
Let $G=(V, E)$ be a finite connected graph, where $V$ and $E$ are the sets of 
vertices and (undirected) edges, respectively. 
We assume that $G$ is simple, that is, that $G$ has no loops nor multiple edges. 
Denoting the set of directed edges by $\overrightarrow{E}$ and defining the 
equivalence relation $\sim$ on $\overrightarrow{E}$ by $(u,v)\sim (v,u)$, 
we regard $E$ as the set of equivalence classes $uv$. 
Thus, $uv=vu$ as elements of $E$. 

Throughout this section, we fix a weight $m_0\colon V\to \R_{> 0}$ 
on the set of vertices $V$, 
and a {\em distance parameter} $d\colon E\to \R_{>0}$ on the set of edges $E$. 
Set $M:=\sum_{u\in V} m_0(u)$ and $D^2 := \sum_{uv\in E} d(uv)^2$. 

We consider the following optimization problem: 
\begin{Problem}\label{problem-l^2-constraint}
Over all maps $\varphi\colon V \to \R^{|V|}$ satisfying 
\begin{align}\label{constraint-l2}
&\sum_{u\in V} m_0(u) \| \varphi(u) \|^2 = M,\\
&\| \varphi(u) - \varphi(v) \| \leq d(uv),\quad \forall uv \in E,\nonumber 
\end{align}
where $\|\cdot \|$ is the Euclidean norm on $\R^{|V|}$, minimize the (squared) 
norm of the affine barycenter
\begin{align*}
\bary(\varphi) = \frac{1}{M} \sum_{u \in V} m_0(u) \varphi(u). 
\end{align*}
In other words, evaluate 
$$
\delta(G,m_0,d) = \inf_\varphi \|\bary(\varphi)\|^2. 
$$
\end{Problem}


Our first observation is that this problem is related to an optimization problem 
regarding the spectral gap of the Laplacian, introduced in \cite{GoringHelmbergWappler2} 
(see also \cite{GoringHelmbergWappler1}) and reviewed below. 
To define the Laplacian, we take a weight $m_1\colon E\rightarrow \R_{\geq 0}$ 
on the set of edges $E$. 
We assume that $G=(V, E')$ is connected, where $E' = \{ uv\in E \mid m_1(uv)> 0 \}$. 
Let $C(V,\R)$ denote the set of functions $\varphi\colon V\to \R$, equipped with 
the inner product defined by $\langle \varphi_1, \varphi_2 \rangle 
= \sum_{u\in V} m_0(u) \varphi_1(u) \varphi_2(u)$. 
Then the Laplacian $\Delta_{(m_0,m_1)}\colon C(V,\R)\to C(V,\R)$ is a nonnegative 
symmetric linear operator, defined by 
\begin{equation}\label{symmetric-normalized-Laplacian}
(\Delta_{(m_0,m_1)} \varphi) (u) = \frac{1}{m_0(u)} \left[ \left( 
\sum_{v\sim u} m_1(uv) \right) \varphi(u) 
- \sum_{v\sim u} m_1(uv) \varphi(v) \right],\quad u\in V, 
\end{equation} 
where we write $v\sim u$ if $uv\in E$. 
Note that $\Delta_{(m_0,m_1)}$ has eigenvalue $0$, and the corresponding eigenspace consists 
precisely of constant functions since $G$ is assumed to be connected. 
Therefore, the second smallest eigenvalue of $\Delta_{(m_0,m_1)}$ is positive; it is 
denoted by $\lambda_1(G, (m_0,m_1))$ and referred to as the first nonzero eigenvalue of 
$\Delta_{(m_0,m_1)}$.  
It is a standard fact that 
$\lambda_1(G, (m_0,m_1))$ is characterized variationally as 
$$
\lambda_1(G, (m_0,m_1)) = \inf \frac{\sum_{uv\in E} m_1(uv) 
( \varphi(u) - \varphi(v) )^2}{\sum_{u\in V} m_0(u) ( \varphi(u)- \overline{\varphi} )^2},  
$$
where $\overline{\varphi} = \sum_{u\in V} m_0(u) \varphi(u)/M$ and the infimum 
is taken over all nonconstant functions $\varphi$. 


\begin{Remark}\label{remark-symmetric-normalized-Laplacian} 
The Laplacian \eqref{symmetric-normalized-Laplacian} 
essentially coincides with the one
employed in \cite{GoringHelmbergWappler2}. 
In fact, write $V=\{ u_1,\dots, u_{|V|} \}$, and define a linear isometry 
$\pi\colon C(V,\R)\to \R^{|V|}$ by $(\pi(\varphi))_i = \sqrt{m_0(u_i)} \varphi(u_i)$ 
for $\varphi\in C(V,\R)$ and $i=1,\dots,|V|$. 
Then 
$$
(\pi\circ \Delta_{(m_0,m_1)}\circ \pi^{-1} (\psi))_i = \left( \frac{1}{m_0(u_i)} 
\sum_{u_j\sim u_i}m_1(u_i,u_j) \right) \psi_i - \sum_{u_j\sim u_i} \frac{m_1(u_i,u_j)} 
{\sqrt{m_0(u_i)m_0(u_j)}}\psi_j.
$$
\end{Remark}


We are ready to state the following 

\begin{Problem}[\cite{GoringHelmbergWappler2}] \label{problemKKNspec} 
Over all weights $m_1$ on $E$, subject to the normalization 
\begin{equation}\label{normalization-edge}
\sum_{uv\in E} m_1(uv)d(uv)^2=D^2,
\end{equation} 
maximize the first nonzero eigenvalue 
$\lambda_1(G,(m_0,m_1))$ of $\Delta_{(m_0,m_1)}$. 
That is, determine 
$$
\sigma(G, m_0,d) := \sup_{m_1} \lambda_1(G, (m_0,m_1)). 
$$
\end{Problem}

\begin{Remark}
When $m_0\equiv 1$ and $d\equiv 1$, $\sigma(G, m_0, d)$ is denoted by $\widehat{a}(G)$ 
and called the {\em absolute algebraic connectivity} of $G$ by Fiedler \cite{Fiedler}. 
\end{Remark}

The following proposition is the key to relating the two optimization problems. 

\begin{Proposition}\label{prop-kondo-inequality}
Let $G=(V,E)$ be a finite connected graph equipped with a vertex-weight $m_0$ 
and a distance parameter $d$. 
For an edge-weight $m_1$ satisfying \eqref{normalization-edge}, we have 
\begin{equation}\label{kondo-inequality2}
\delta(G,m_0,d)\geq 1 - \frac{D^2/M}{\lambda_1(G,(m_0,m_1))}. 
\end{equation}
In \eqref{kondo-inequality2}, the equality sign holds if and only if there exists 
$\varphi\colon V\to \R^{|V|}$ satisfying \eqref{constraint-l2} such that 
\begin{enumerate}
\renewcommand{\theenumi}{\roman{enumi}}
\renewcommand{\labelenumi}{(\theenumi)}
\item $m_1(uv)(d(uv)^2 - \| \varphi(u)-\varphi(v)\|^2)=0,\,\, \forall uv\in E$, 
\item $\Delta_{(m_0,m_1)}\varphi = \lambda_1(G,(m_0,m_1)) (\varphi - \bary(\varphi))$, 
that is, each component of the map $\varphi - \bary(\varphi)$ is an eigenvector 
of the eigenvalue $\lambda_1(G,(m_0,m_1))$ of the Laplacian $\Delta_{(m_0,m_1)}$. 
\end{enumerate}
\end{Proposition}

\begin{proof}
\begin{eqnarray*}
\| \bary(\varphi) \|^2 &=& \frac{1}{M} \sum_{u\in V} m_0(u) 
\| \varphi(u) \|^2 
- \frac{1}{M} \sum_{u\in V} m_0(u) \| \varphi(u) - \bary(\varphi) 
\|^2\\ 
&\geq& \frac{1}{M} \sum_{u\in V} m_0(u) \| \varphi(u) \|^2\\ 
&& - \frac{1}{M} \frac{1}{\lambda_1(G,(m_0,m_1))} \sum_{uv\in E} 
m_1(uv) \| \varphi(u)-\varphi(v) \|^2. 
\end{eqnarray*}
Since $\varphi$ obeys the constraints \eqref{constraint-l2}, the rightmost 
expression is 
\begin{eqnarray*}
&\geq& 1 - \frac{1}{M} \frac{1}{\lambda_1(G,(m_0,m_1))} \sum_{uv\in E} 
m_1(uv) d(uv)^2\\ 
&=& 1 - \frac{1}{M}\frac{D^2}{\,\lambda_1(G,(m_0,m_1))}. 
\end{eqnarray*}

The assertion on the equality case is clear.  
\end{proof}

\medskip
Since the left-hand sides of \eqref{kondo-inequality2} do not depend on $m_1$, 
we obtain 

\begin{Corollary}\label{cor-kondo-inequality}
Let $G=(V,E)$ be a finite connected graph equipped with a vertex-weight $m_0$ 
and a distance parameter $d$. 
Then we have 
\begin{equation}\label{kondo-inequality2'}
\delta(G,m_0,d)\geq 1 - \frac{D^2/M}{\sigma(G, m_0,d)}. 
\end{equation}
In \eqref{kondo-inequality2'}, the equality sign holds if and only if there exist 
an edge-weight $m_1$ and $\varphi\colon V\to \R^{|V|}$ satisfying \eqref{constraint-l2} 
and the two conditions (i), (ii) as in the statement of Proposition \ref{prop-kondo-inequality}. 
\end{Corollary}

\begin{Remark}
The conditions for the equality case in Proposition \ref{prop-kondo-inequality} 
and Corollary \ref{cor-kondo-inequality} coincide with the so-called KTT conditions 
associated with Problems \ref{problem-l^2-constraint} and \ref{problemKKNspec} 
which are shown to be dual to each other in \S 3. 
\end{Remark}



\begin{Example}
Let $G_p$ be the incidence graph of the projective plane ${\bf P}^2(\F_p)$ over the field 
$\F_p=\Z/p\Z$, where $p$ is a prime number. 
Since ${\bf P}^2(\F_p)$ has $p^2+p+1$ lines and $p^2+p+1$ points with $p+1$ points on every line 
and $p+1$ lines through every point, $G_p$ is a $(p+1)$-regular bipartite graph with $2(p^2+p+1)$ 
vertices. Note also that $G_p$ has diameter $3$. 
Define weights $m_0, m_1$ and a distance parameter $d$ by 
\begin{align*}
& m_0(u) = p+1,\quad \forall u\in V,\\ 
& m_1(uv) = 1,\, d(uv) =1\quad \forall uv \in E,
\end{align*} 
so that the normalization \eqref{normalization-edge} is satisfied and $D^2/M=1/2$. 
By a result of Feit and Higman \cite{FeitHigman}, we have $\lambda_1(G_p,(m_0,m_1)) 
= 1 - \frac{\sqrt{p}}{p+1}$, and therefore 
$$
1 - \frac{D^2/M}{\lambda_1(G_p,(m_0,m_1)) } = \frac{p+1-2\sqrt{p}}{2(p+1-\sqrt{p})}.
$$
On the other hand, Problem \ref{problem-l^2-constraint} for $G_p$ is solved in 
\cite{IzekiNayatani}, and the solution $\varphi$ satisfies 
$$
\langle \varphi(u), \varphi(v) \rangle = \left\{\begin{array}{ccc} 
\frac{1}{2} &\mbox{if}& d_{G_p}(u,v)=1,\\ 
\frac{p-1-\sqrt{p}}{2p} &\mbox{if}& d_{G_p}(u,v)=2,\\ 
\frac{p^2-p-(p+1)\sqrt{p}}{2p^2} &\mbox{if}& d_{G_p}(u,v)=3, 
\end{array} \right. 
$$
where $d_{G_p}$ is the combinatorial distance on $V$. 
It follows that 
$$
\delta(G,m_0,d) = \left\| \frac{1}{M} \sum_{u\in V} m_0(u) \varphi(u) \right\|^2 
= \frac{p^2+1-(p+1)\sqrt{p}}{2(p^2+p+1)} 
= \frac{p+1-2\sqrt{p}}{2(p+1-\sqrt{p})}. 
$$
Thus the equality sign holds in \eqref{kondo-inequality2} (and hence in \eqref{kondo-inequality2'}). 
In particular, when the vertex-weight $m_0\equiv p+1$ and the distance parameter $d\equiv 1$ 
are fixed, the choice of edge-weight $m_1\equiv 1$ maximizes the spectral gap 
$\lambda_1(G_p,(m_0,m_1))$ among all those subject to the normalization \eqref{normalization-edge}, 
and $\sigma(G, m_0, d) = 1 - \frac{\sqrt{p}}{p+1}$. 
\end{Example}

\section{Relation to other optimization problems}
In \cite{GoringHelmbergWappler2} (see also \cite{GoringHelmbergWappler1}) 
an optimization problem similar to Problem \ref{problem-l^2-constraint} is considered. 
Again, the problem is concerned with graph-embeddings, and very importantly it is dual to 
Problem \ref{problemKKNspec}. 
In this section, after reviewing this duality, we discuss how Problem \ref{problem-l^2-constraint} 
is related to  the one in \cite{GoringHelmbergWappler2}. 
(In fact, our Problem \ref{problem-l^2-constraint} is also dual to Problem \ref{problemKKNspec}.  
This will be discussed in \S 3.) 

Let $G=(V,E)$ be a finite connected graph equipped with a vertex-weight 
$m_0\colon V\rightarrow \R_{> 0}$ and a distance parameter $d\colon E\to \R_{>0}$. 

\begin{Problem}[\cite{GoringHelmbergWappler2}]\label{problemGHWemb}
Over all maps $\varphi\colon V \to \R^{|V|}$ satisfying 
\begin{align}\label{constraintGHW}
&		\sum_{u\in V} m_0(u) \varphi(u) = 0,\\
&		\| \varphi(u) - \varphi(v) \| \leq d(uv),\quad \forall uv \in E,\nonumber 
\end{align}
maximize 
\begin{align*}
\frac{1}{M}\sum_{u \in V} m_0(u) \| \varphi(u) \|^2. 
\end{align*}
That is, evaluate 
$$
\nu(G,m_0,d) := \sup_{\varphi} \frac{1}{M}\sum_{u \in V} m_0(u) \| \varphi(u) \|^2. 
$$
\end{Problem}


It is shown in \cite{GoringHelmbergWappler1} that Problem \ref{problemGHWemb} is dual to 
Problem \ref{problemKKNspec}.
For the precise formulation of this duality, we refer the reader to 
\cite[pp.\! 474-475]{GoringHelmbergWappler1}. 
By semidefinite duality theory together with strict feasibility, they deduce that the optimal 
values of the two problems (are attained and) coincide. 
We record this fact as 

\begin{Theorem}[\cite{GoringHelmbergWappler1}]\label{theoremGHW}
For any finite connected graph $G=(V,E)$ equipped with a vertex-weight 
$m_0\colon V\rightarrow \R_{> 0}$ and a distance parameter $d\colon E\to \R_{>0}$, 
we have
\begin{equation}\label{duality}
\nu(G,m_0,d) = \frac{D^2/M}{\sigma(G,m_0,d)}.
\end{equation}
\end{Theorem}

\begin{Remark}\label{inequality_part_of_theoremGHW}
The inequality 
\begin{equation}\label{inequality_part_of_duality}
\nu(G,m_0,d) \leq \frac{D^2/M}{\sigma(G,m_0,d)} 
\end{equation}
is an analogue of \eqref{kondo-inequality2'} and can be proved by a similar argument. 
Indeed, if $\varphi\colon V \to \R^{|V|}$ is a map satisfying the constraints 
\eqref{constraintGHW}, then 
\begin{eqnarray}\label{emb_spec}
\sum_{u\in V} m_0(u) \| \varphi(u) \|^2 &=&  \sum_{u\in V} m_0(u) \| \varphi(u) - \overline{\varphi} \|^2\\ 
&\leq& \frac{1}{\lambda_1(G,(m_0,m_1))} \sum_{uv\in E} m_1(uv) \| \varphi(u)-\varphi(v) \|^2
\nonumber\\ 
&\leq& \frac{D^2}{\lambda_1(G,(m_0,m_1))}.\nonumber 
\end{eqnarray}
Therefore, \eqref{inequality_part_of_duality} follows. 

Let $m_1$ and $\varphi$ be optimal solutions for Problems \ref{problemKKNspec} 
and \ref{problemGHWemb}, respectively.
Then the inequality signs in \eqref{emb_spec} become equalities, and hence each component 
of $\varphi$ has to be an eigenvector of the eigenvalue $\lambda_1(G,(m_0,m_1))$ 
of $\Delta_{(m_0,m_1)}$. 
This verifies Remark 3.3 on p.\! 292 of \cite{GoringHelmbergWappler1}. 
\end{Remark}

By combining \eqref{kondo-inequality2'} and \eqref{duality}, we obtain 
$$
\delta(G,m_0,d)\geq 1 - \frac{D^2/M}{\sigma(G, m_0,d)} = 1 - \nu(G,m_0,d). 
$$
The following proposition gives a more precise relation between 
Problems \ref{problem-l^2-constraint} and \ref{problemGHWemb} concerning optimal embeddings. 

\begin{Proposition}\label{relation_prob_emb}
For any finite connected graph $G=(V,E)$ equipped with a vertex-weight 
$m_0\colon V\rightarrow \R_{> 0}$, 
we have
\begin{equation}\label{formula_prob_emb}
\delta(G.m_0,d) = \max\left\{ 1 - \nu(G,m_0,d), 0 \right\}.
\end{equation}
\end{Proposition}

\begin{proof}
Let $\varphi$ be an optimal solution of Problem \ref{problem-l^2-constraint}. 
Then $\psi=\varphi - \overline{\varphi}$ satisfies the constraints \eqref{constraintGHW}
of Problem \ref{problemGHWemb}. 
Since
\begin{eqnarray*}
\sum_{u \in V} m_0(u) \| \psi (u) \|^2 &=& \sum_{u \in V} m_0(u) \| \varphi(u) \|^2 
- M \|\overline{\varphi} \|^2\\
&=& M ( 1 - \delta(G,m_0,d) ), 
\end{eqnarray*}
we obtain 
$$
\nu(G,m_0,d)\geq 1 - \delta(G,m_0,d),\,\, \mbox{or}\,\,\, \delta(G,m_0,d) \geq 1 - \nu(G,m_0,d). 
$$

The other way around, let $\varphi$ be an optimal solution of Problem \ref{problemGHWemb}. 
We treat the following two cases separately: (i) $\nu(G,m_0,d)>1$, 
(ii) $\nu(G,m_0,d)\leq 1$. 
In case (i), 
$$
\psi = \sqrt{1 /\nu(G,m_0,d)}\, \varphi
$$
satisfies the constraints \eqref{constraint-l2} of Problem \ref{problem-l^2-constraint}.
Since 
$$
\left\|\frac{1}{M}\sum_{u\in V} m_0(u) \psi (u)\right\|^2 
= \frac{1}{M^2 \nu(G,m_0,d)} \left\|\sum_{u\in V} m_0(u) \varphi(u)\right\|^2 = 0, 
$$
we obtain $\delta(G,m_0,d) = 0$. 
In case (ii), define $\psi$ by 
$$ 
\psi(u) = \varphi(u) + \sqrt{1 - \nu(G,m_0,d))}\, e,\quad 
u\in V,
$$
where $e$ is any unit vector in $\R^{|V|}$. 
Then $\psi$ satisfies the constraints \eqref{constraint-l2} of Problem \ref{problem-l^2-constraint}, and 
$$
\left\|\frac{1}{M}\sum_{u\in V} m_0(u) \psi (u)\right\|^2 
= 1 - \nu(G,m_0,d). 
$$
Therefore, 
$$
\delta(G,m_0,d)\leq 1 - \nu(G,m_0,d). 
$$
We may now conclude \eqref{formula_prob_emb}. 
\end{proof}

Combining Proposition \ref{relation_prob_emb} with Theorem \ref{theoremGHW}, 
we obtain the following 

\begin{Corollary}\label{cor-kondo-inequality_improved}
Let $G=(V,E)$ be a finite connected graph equipped with a vertex-weight $m_0$ 
and a distance parameter $d$. 
Then we have 
\begin{equation}\label{formula_prob_emb}
\delta(G.m_0,d) = \max\left\{ 1 - \frac{D^2/M}{\sigma(G,m_0,d)}, 0 \right\}.
\end{equation}
\end{Corollary}

Notice that \eqref{formula_prob_emb} improves the inequality \eqref{kondo-inequality2'} 
of Corollary \ref{cor-kondo-inequality}.

\section{Optimal embeddings of semi-regular polyhedra}
In this section, we consider graphs isomorphic to the one-skeltons of regular and semi-regular polyhedra, and decide their optimal embeddings for Problem \ref{problemGHWemb}. 
It will turn out that the resulting embeddings obtained as the optimal solutions of Problem 
\ref{problemGHWemb} coincide with those realizing the graphs as one-skeltons of the given polyhedra. 

\subsection{Platonic solids}

The Platonic solids are the five regular convex polyhedra:
the regular tetrahedron, the regular hexahedron, the regular octahedron, the regular dodecahedron and the regular icosahedron. 

We discuss the dodecahedron in detail. 
The other polyhedra can be handled similarly.
Let $C_{20} = (V,E)$ be a graph isomorphic to the one-skelton of the dodecahedron, 
which has $20$ vertices and $30$ edges. 
Let parameters $m_0$, $d$ be uniform ones: $m_0\equiv 1$, $d\equiv 1$.
We verify that the optimal embedding of $C_{20}$ realizes it as the one-skelton of the regular 
dodecahedron. 
In fact, if we choose $m_1$ uniform, that is, $m_1\equiv 1$, then the first nonzero eigenvalue 
of the corresponding Laplacian is computed as $\lambda_2 (C_{20}, (m_0, m_1)) = 3 - \sqrt{5}$. 

On the other hand, for the regular dodecahedron with edge length one, the radius of its 
circumscribed sphere is $(\sqrt{15} + \sqrt{3})/4$. 
Therefore, this feasible solution has  
$30/[20((\sqrt{15} + \sqrt{3})/4)^2] = 3 - \sqrt{5}$, the same value as above, 
as the objective value of the embedding problem. 
Thus we conclude that the optimal embedding of $C_{20}$ gives the one-skelton of the regular dodecahedron. 

Similar results are obtained for the other four regular polyhedra.
The optimal values of Problem \ref{problemKKNspec} for these polyhedra with the same choices 
of parameters are listed in Table \ref{maxspectralgap_Platonicsolids}.

\begin{table}[htb]
 \caption{Maximum spectral gaps for the Platonic solids}
\label{maxspectralgap_Platonicsolids}
\centering
  \begin{tabular}{|c|c|} \hline
Regular polyhedron & Maximum spectral gap \\ \hline \hline
Tetrahedron & $4$ \\ \hline
Hexahedron & $2$ \\ \hline
Octahedron & $4$ \\ \hline
Dodecahedron & $3 - \sqrt{5}$ \\ \hline
Icosahedron & $5 - \sqrt{5}$ \\ \hline
  \end{tabular}
\end{table}

\subsection{Fullerene $C_{60}$}
Let $C_{60} = (V,E)$ denote a graph isomorphic to the one-skelton of a truncated 
icosahedron which is also called a buckyball. 
$C_{60}$ has 60 vertices and $90$ edges, and $60$ of the edges are pentagonal edges 
and $30$ of them are hexagonal ones. 
Here, an edge is called {\em pentagonal} it it is on the boundary of a pentagonal face; 
otherwise, it is called {\em hexagonal}. 
Let the vertex weight $m_0$ be the uniform one: $m_0\equiv 1$. 
Choose the edge weight $m_1$ as 
\begin{eqnarray*}
m_1 (uv) = \left\{
 \begin{array}{ll}
x, \quad  \text{if} \ uv \ \text{is a pentagonal edge}, \\
y, \quad  \text{if} \ uv \ \text{is a hexagonal edge}. 
 \end{array}
 \right.
\end{eqnarray*}
Then by a result of \cite{Chung}, the first nonzero eigenvalue of the Laplacian 
for the above vertex and edge weights is 
\begin{eqnarray*}
\lambda_1 (G, (m_0, m_1)) &=& (2x+y) \\
& & -\frac{x}{4} \left( 3 + \sqrt{5} + \sqrt{2} \sqrt{15-5\sqrt{5}-4t+4\sqrt{5}t+8t^2} \right) \Bigl|_{t=\frac{y}{x}}.
\end{eqnarray*}

We begin with the case that the edge parameter $d$ is uniform: $d\equiv 1$. 
The circumscribed sphere of the truncated icosahedron with edge length one 
has radius $\sqrt{58+18\sqrt{5}}/4$.
Therefore, the objective value of the problem (\ref{problemGHWemb}) for this 
embedding is 
$$
60\left(\frac{\sqrt{58+18\sqrt{5}}}{4}\right)^2=\frac{15}{2} (29+9\sqrt{5}).
$$
On the other hand, the choice of $m_1$ with 
\begin{eqnarray*}
x = \frac{1}{218} (189 + 9 \sqrt{5}), \quad y = \frac{1}{109} (138 - 9 \sqrt{5})
\end{eqnarray*}
satisfies the normalization \eqref{normalization-edge} of Problem \ref{problemKKNspec}. 
The objective value for this feasible solution is $(87-27\sqrt{5})/109$, and 
$$
\frac{D^2/M}{(87-27\sqrt{5})/109}
=\frac{15}{2} (29+9\sqrt{5}).
$$ 
Therefore, the one-skelton of the truncated icosahedron is realized by an optimal embedding.

We now consider the case that the distance parameter $d$ is given by 
\begin{eqnarray*}
d(uv) = \left\{
 \begin{array}{ll}
a, \quad \text{if} \ uv \ \text{is a pentagonal edge}, \\
b, \quad \text{if} \ uv \ \text{is a hexagonal edge}.
 \end{array}
 \right.
\end{eqnarray*}
It is reasonable to expect that the one-skelton of the truncated icosahedron in which the ratio 
of the length of a pentagonal edge to that of a hexagonal edge is $a : b$ is obtained as an 
optimal embedding.
The barycenter of this truncated icosahedron is at the origin again, and the objective value 
for this feasible solution is
\begin{equation} \label{objective value of C_{60}}
\frac{15}{2} a^2 \left\{ (5+\sqrt{5})s^2 +(4\sqrt{5} + 12)(s + 1) \right\},
\end{equation}
where $s= b/a$. 
(Note that this value coincides with the one in the previous case that $a=b=1$.) 

A feasible solution for Problem \ref{problemKKNspec} with the parameter $d$ is found as
\begin{eqnarray*}
x = \frac{\left( 2a^2 + b^2 \right) \left( (6 + 2 \sqrt{5}) a + (3 + \sqrt{5}) b \right)}{a \left( \left( 12+ 4 \sqrt{5} \right) a^2 + \left( 12+ 4 \sqrt{5} \right) ab + \left( 5 +\sqrt{5} \right) b^2 \right)} , \\
y = \frac{\left( 2a^2 + b^2 \right) \left( (6 + 2 \sqrt{5}) a + (5 + \sqrt{5}) b \right)}{b \left( \left( 12+ 4 \sqrt{5} \right) a^2 + \left( 12+ 4 \sqrt{5} \right) ab + \left( 5 +\sqrt{5} \right) b^2 \right)}.
\end{eqnarray*}
The objective value for this feasible solution is
$$A := \frac{4(2a^2 + b^2)}{\left( 12+ 4 \sqrt{5} \right) a^2 + \left( 12+ 4 \sqrt{5} \right) ab + \left( 5 +\sqrt{5} \right) b^2},$$
and
$$
\frac{D^2/M}{A}=\frac{15}{2} a^2 \left\{ (5+\sqrt{5})s^2 +(4\sqrt{5} + 12)(s + 1) \right\}.
$$ 
Since the objective values coincide, we get the expected result.

\subsection{Other Archimedean solids}
Archimedean solids are convex polyhedra all of whose faces are regular polygons, 
and which have a symmetry group acting transitively on the vertices. 
(Note, however, that the prisms, antiprisms and five Platonic solids are excluded.) 
Archimedean solids are classified and identified by the vertex configuration which refers to polygons 
that meet at any vertex. 
For example, a truncated icosahedron is denoted by $(5,6,6)$.

Let $G$ be the one-skelton of a truncated icosidodecahedron $(4,6,10)$.
And let an edge weight $m_1$ be given by 
\begin{eqnarray*}
m_1 (uv) &=& \left\{
 \begin{array}{ll}
x, \quad \text{if $uv$ separates $4$- and $6$-gons}, \\
y, \quad \text{if $uv$ separates $4$- and $10$-gons}, \\
z, \quad \text{if $uv$ separates $6$- and $10$-gons}, 
 \end{array}
 \right.
\end{eqnarray*}
where $x,y,z$ satisfy $x+y+z=1$. 
In \cite{IvrissimtzisPeyerimhoff} the optimization problem minimizing the second largest eigenvalue 
of the weighted adjacency matrix over all edge weights $m_1$ of the above form is solved, and 
$(179+24\sqrt{5})/241$ is obtained as the optimal value.
By choosing parameters $m_0\equiv 1$ and $d\equiv \sqrt{3}$, edge weights $m_1$ of the above form 
satisfies the normalization \eqref{normalization-edge} of Problem \ref{problemKKNspec}. 
Thus we have 
$$\frac{|E|}{\sigma(G,m_0,d)} \leq \frac{180}{1- (179+24\sqrt{5})/241} = 90 (31 + 12 \sqrt{5}).$$
For the truncated icosidodecahedron with side length $\sqrt{3}$, the radius of its circumscribed sphere is $\sqrt{93 + 36 \sqrt{5}}/2$, and thus the objective value for Problem \ref{problemGHWemb} is $120 \times (93 + 36 \sqrt{5})/4 = 90 (31 + 12 \sqrt{5})$.
Therefore, the one-skelton of the truncated icosidodecahedron is realized by an optimal embedding.

In the same way, the one-skeltons of the truncated cuboctahedron $(4,6,8)$ and the truncated octahedron $(4,6,6)$ 
are also realized by optimal embeddings of the corresponding graphs.

\section{Duality between Problem \ref{constraint-l2} and Problem \ref{problemKKNspec}}
In \cite{GoringHelmbergWappler2} it is shown by using the Lagrange approach that Problem \ref{problemKKNspec} is dual to Problem \ref{problemGHWemb}. 
In this section, we show that Problem \ref{problemKKNspec} is also dual to Problem \ref{problem-l^2-constraint}. 

Let $\varphi\colon V \to \R^{|V|}$ be an arbitrary map which are unconstrained, and let $\widetilde{m}_1 \colon E \rightarrow \R_{\geq 0}$ and $\mu \in \mathbb{R}$ be new variables.
We define the Lagrange function by  
\begin{eqnarray}
L(\widetilde{m}_1, \mu, \varphi) &=& \sum_{uv \in E} \widetilde{m}_1 (uv) 
\left( || \varphi (u)-\varphi (v) ||^2 - d(uv)^2 \right) \nonumber \\
&& + \mu \sum_{u \in V} m_0 (u) \left( || \varphi (u) ||^2 - 1 \right) 
+ \left\| \sum_{u \in V} m_0 (u) \varphi (u) \right\|^2.
\label{Lagrangef_ourdual}
\end{eqnarray}
It is easy to see that the following inequality holds.
\begin{eqnarray*}
\inf_{\varphi} \ \sup_{\widetilde{m}_1,\mu} L(\widetilde{m}_1, \mu, \varphi) \geq  \sup_{\widetilde{m}_1,\mu} \ \inf_{\varphi} L(\widetilde{m}_1, \mu, \varphi).
\end{eqnarray*}
For any $\varphi$ we have 
\begin{eqnarray*}
\sup_{\substack{\widetilde{m}_1 \colon E \rightarrow \R_{\geq 0}, \\ \mu \in \mathbb{R}}} L(\widetilde{m}_1, \mu, \varphi) = 
\left\{
\begin{array}{ll}\renewcommand{\arraystretch}{1.5}
\left\| \sum_{u \in V} m_0 (u) \varphi (u) \right\|^2 & \text{if} \ || \varphi (u)-\varphi (v) || \leq d(uv), \ \forall uv \in E \\
& \qquad \text{and} \ \sum_{u\in V} m_0(u) \| \varphi(u) \|^2 = M, \\
\infty & \text{otherwise}.
\end{array}
\right.
\end{eqnarray*}
Thus the optimization system of the left-hand side is the same as that of Problem \ref{problem-l^2-constraint}, that is,
$$M^2 \, \delta (G, m_0,d) =\inf_{\varphi \ \text{satisfying (\ref{constraint-l2})}} \ \sup_{\widetilde{m}_1,\mu} L(\widetilde{m}_1, \mu, \varphi).$$
The right-hand side gives its dual problem, which we shall identify. 
To do so, we rewrite the Lagrange function (\ref{Lagrangef_ourdual}) as
\begin{eqnarray*} 
L(\widetilde{m}_1, \mu, \varphi) &=& -\mu M -\sum_{uv \in E} d(uv)^2 \widetilde{m}_1 (uv) \\
&& + \left\| \sum_{u \in V} m_0 (u) \varphi (u) \right\|^2 + \mu \sum_{u \in V} m_0 (u) || \varphi (u) ||^2 \\
&& + \sum_{uv \in E} \widetilde{m}_1 (uv) || \varphi (u)-\varphi (v) ||^2.
\end{eqnarray*}
Let $\mu \in \mathbb{R}$ and $\widetilde{m}_1 \colon E \rightarrow \R_{\geq 0}$. If these parameters satisfy the inequality
\begin{equation}\label{Lagrange-dual-inequality}
\left\| \sum_{u \in V} m_0 (u) \varphi (u) \right\|^2 + \mu \sum_{u \in V} m_0 (u) 
|| \varphi (u) ||^2 + \sum_{uv \in E} \widetilde{m}_1 (uv) || \varphi (u)-\varphi (v) ||^2 \geq 0 
\end{equation}
for all $\varphi$, 
then the minimum of $L(\widetilde{m}_1, \mu, \varphi)$ over $\varphi$ is attained 
when $\varphi \equiv 0$.
Otherwise, $L(\widetilde{m}_1, \mu, \varphi)$ diverges to negative infinity:
\begin{eqnarray*}
\underset{\varphi}{\inf} \ L(\widetilde{m}_1, \mu, \varphi) = 
\left\{
\begin{array}{ll}\renewcommand{\arraystretch}{1.5}
-\mu M -\sum_{uv \in E} d(uv)^2 \widetilde{m}_1 (uv) & \text{if} \, \varphi \, \text{satisfies the inequality (\ref{Lagrange-dual-inequality})}, \\
-\infty & \text{otherwise}.
\end{array}
\right.
\end{eqnarray*}
We derive $\lambda_1 (G, (m_0,\widetilde{m}_1))$ from the inequality (\ref{Lagrange-dual-inequality}). 
If $\varphi$ is a constant map, then the inequality (\ref{Lagrange-dual-inequality}) becomes
\begin{eqnarray*} 
0 &\leq& \left\| \sum_{u \in V} m_0 (u) \varphi (u) \right\|^2 + \mu \sum_{u \in V} m_0 (u) || \varphi (u) ||^2 \\
&=& \left( M + \mu \right) \sum_{u \in V} m_0 (u) || \varphi (u) ||^2. 
\end{eqnarray*}
Thus we get $M \geq -\mu$.

Next we assume $\varphi$ is an eigenmap of $\lambda_1 (G, (m_0,\widetilde{m}_1))$.
Then the inequality (\ref{Lagrange-dual-inequality}) is
\begin{eqnarray*} 
0 &\leq& M^2 || \text{bar} (\varphi) ||^2 + \mu \sum_{u \in V} m_0 (u) || \varphi (u) ||^2 \\
&& + \lambda_1 (G, (m_0,\widetilde{m}_1)) \left( \sum_{u \in V} m_0 (u) || \varphi (u) ||^2 - M || \text{bar} (\varphi) ||^2 \right).
\end{eqnarray*}
By using $\text{bar} (\varphi) = 0$ we get $\lambda_1 (G, (m_0,\widetilde{m}_1)) \geq -\mu$. 

Therefore the dual problem is a problem that maximizes
$$-\mu M -\sum_{uv \in E} d(uv)^2 \widetilde{m}_1 (uv)$$
over all $\mu$ and $\widetilde{m}_1$ subject to the constraints $M \geq -\mu$ and $\lambda_1 (G, (m_0,\widetilde{m}_1)) \geq -\mu$.

$-\mu$ can be replaced by $\mu$.
Introducing a new variable $\lambda >0$,  
we may add a new constraint $\sum_{uv \in E} d(uv)^2 \widetilde{m}_1 (uv) = 1/\lambda$.
Then the objective function is $\mu M -1/\lambda$, and all constraints are listed as
\begin{eqnarray*}
&M \geq \mu, \\
&\lambda_1 (G, (m_0,\widetilde{m}_1)) \geq \mu, \\
&\sum_{uv \in E} d(uv)^2 \widetilde{m}_1 (uv) = \frac{1}{\lambda}.
\end{eqnarray*}
If we set $m_1 (uv) := D^2 \lambda \, \widetilde{m}_1 (uv)$ for $uv \in E$, 
then the constraints are
\begin{eqnarray*} 
& M \geq \mu, \\
& -\frac{1}{\lambda} \leq -\frac{1}{\lambda_1 (G, (m_0, m_1))} \mu D^2, \\
& \sum_{uv \in E} d(uv)^2 m_1 (uv) = D^2.
\end{eqnarray*}

In this optimization process, we first optimize the objective function with respect to the 
parameters $\mu$ and $\lambda$. 
Thus $\mu$ attains $M$ and $-1/\lambda$ attains $-\mu D^2/\lambda_1 (G, (m_0, m_1))$, 
and the problem reduces to the following: Maximize 
$$M^2 -\frac{D^2 M}{\lambda_1 (G, (m_0, m_1))}$$
over all edge weight $m_1 \colon E \rightarrow \R_{\geq 0}$ subject to $\sum_{uv \in E} d(uv)^2 m_1 (uv) = D^2$.
This problem is nothing but Problem \ref{problemKKNspec} and the desired duality is established. 
In particular, the inequality (\ref{kondo-inequality2'}) in Corollary \ref{cor-kondo-inequality} 
is reproduced.

\end{document}